\documentclass{amsart}
\newtheorem{theorem}{Theorem}[section]
\newtheorem{lemma}[theorem]{Lemma}
\newtheorem{cor}[theorem]{Corollary}
\theoremstyle{definition}
\newtheorem{definition}[theorem]{Definition}

\theoremstyle{remark}

\numberwithin{equation}{section}

\begin{document}
	\title{Commuting Jordan derivations over a Ring with idempotents}
	\author{Sk Aziz}
	\address{Department of Mathematics, Indian Institute of Technology Patna, Patna-801106}
	\curraddr{}
	\email{E-mail: aziz$\_$2021ma22@iitp.ac.in}
	\thanks{}
	\author{Om Prakash$^{\star}$}
	\address{Department of Mathematics, Indian Institute of Technology Patna, Patna-801106}
	\curraddr{}
	\email{om@iitp.ac.in}
	\thanks{* Corresponding author}
	\author{Arindam Ghosh}
	\address{Department of Mathematics, Government Polytechnic Kishanganj, Kishanganj-855116}
	\curraddr{}
	\email{E-mail: arindam.rkmrc@gmail.com}
	\thanks{}

	
	
	\dedicatory{}
	
	\maketitle
	\begin{abstract}
		This paper explores the behavior of commuting Jordan derivations over prime rings with non-trivial idempotents and demonstrates that they become zero maps. Further, it establishes this result for commuting Jordan higher derivations over prime rings under specific conditions. In fact, it introduces commuting generalized Jordan derivations over rings and establishes that the zero map is the only commuting generalized Jordan derivation over prime rings under certain assumptions.
	\end{abstract}
	
	\section{Introduction}
	
	Throughout the paper, $P$ represents a prime ring with center $Z(P)$ unless otherwise stated. Also, $[p_1, p_2]$ represents the commutator of $p_1$ and $p_2$, defined as $p_1p_2 - p_2p_1$, for all $p_1, p_2 \in P$. Recall that a ring $P$ is said to be a prime ring if $p_1 P p_2 = 0$ implies that $p_1=0$ or $p_2=0$. Further, a ring $P$ is said to be $2$-torsion free when $2p = 0$ for some $p\in P$ implies $p = 0$. An additive mapping $h: P \rightarrow P$ is called a derivation if $h(p_1p_2) = h(p_1)p_2 + p_1h(p_2)$ for all $p_1, p_2 \in P$. On the other hand, a Jordan derivation $d: P \rightarrow P$ is an additive map satisfying $d(p^2) = d(p)p + pd(p)$ for all $p \in P$. An additive map $D$ is called a generalized Jordan derivation if $D(p^2) = D(p)p + pd(p)$ for all $p \in P$ and for some Jordan derivation $d$ on $P$. A left centralizer $f: P \rightarrow P$ is an additive map satisfying $f(p_1 p_2) = f(p_1) p_2 $ for all $p_1, p_2 \in P$. Similarly, we define the right centralizer. An additive map $g:P \rightarrow P$ is said to be a Jordan left centralizer if $g(p^2)=g(p)p$, for all $p\in P$. Similarly, a Jordan right centralizer could be defined. A map $F: P\rightarrow P$ is called a centralizing map if $[F(p), p] \in Z(P)$ for all $p \in P$. In particular, if $[F(p), p] = 0$ for all $p \in P$, then $F$ is called a commuting map on $P$.\\
	
	Herstein \cite{hers} initially established that over a $2$-torsion-free prime ring, every Jordan derivation becomes a derivation. Later, Cusack \cite{cu} extended this result to semiprime rings and provided an alternative proof as well. In $2008$, Vukman \cite{vuk1} discussed left Jordan derivations over semiprime rings. In his work, he demonstrated that if $P$ is a semiprime ring which is $2$-torsion free and $d$ is a left Jordan derivation over $P$, then $d$ becomes a derivation that maps $P$ into $Z(P)$. Recently, this investigation has gone on to the characterizations of Jordan derivations, specifically over triangular rings. Notably, a study of Fo{\v{s}}ner \cite{fos} established that every Jordan derivation over a triangular ring (may not contain unity) is indeed a derivation under certain conditions. For further exploration of the topic and additional insights into Jordan derivations,  we refer \cite{bre2,che,div,ghar,vuk2,zha, reh, kam}.\\
	In recent decades, the investigation of commuting maps has emerged as a highly active area of research within the study of mappings on rings. For detailed information on commuting maps, readers can see \cite{bre1}. The study of commuting and centralizing mappings dates back to 1955, marked by Divinsky's work \cite{div}. Divinsky demonstrated that a simple Artinian ring is commutative when a nontrivial automorphism satisfies some commuting properties. Later, Posner \cite{pos} made a significant contribution in $1957$ by establishing the initial result concerning commuting derivations. Posner's theorem asserts that if d is a commuting derivation acting on a prime ring $P$, then $P$ is a commutative ring, or the derivation $d$ is identically zero. In $2023$, Hosseini and Jing \cite{hos} proved that every commuting Jordan derivation over a triangular ring (may not contain unity) is identically zero map. \\
	
	In this paper, our primary objective is to investigate commuting Jordan derivations on a $2$-torsion free ring $P$. We present a distinct result from Posner's result with a completely different type of proof by demonstrating that any commuting Jordan derivation over a $ 2$-torsion-free prime ring $P$ must be identically zero. Moreover, we prove that commuting Jordan left centralizers over a $ 2$-torsion free prime ring under some specific conditions is basically a zero map. As a consequence, we prove that every commuting generalized Jordan derivation over the same ring is identically zero, where we define commuting generalized Jordan derivation in the following way:
	
	\begin{definition}
		Let $P$ be a ring. An additive map $D:P \rightarrow P$ is said to be a commuting generalized Jordan derivation if
		\begin{align*}
			& D(p^2)=D(p)p+pd(p)~\text{and}~D(p)p=pD(p),
		\end{align*}
		\text{for all}~$p\in P$ \text{and $d$ is a Jordan derivation}~ over $P$.	
	\end{definition}
	Commuting Jordan derivation and commuting Jordan left centralizer are defined in \cite{hos} as the follows:
	\begin{definition}
		An additive map $d:P \rightarrow P$ is said to be a commuting Jordan derivation if
		\begin{align*}
			& d(p^2)=d(p)p+pd(p)~\text{and}~d(p)p=pd(p),~\text{for all}~p\in P.
		\end{align*}
	\end{definition}
	
	\begin{definition}
		An additive map $D:P \rightarrow P$ is said to be a commuting Jordan left centralizer if
		\begin{align*}
			& D(p^2)=D(p)p~\text{and}~D(p)p=pD(p),~\text{for all}~p\in P.
		\end{align*}
	\end{definition}
	
	Let $e_1$ be a nontrivial idempotent of $P$ and $e_2=1-e_1$ where $1$ is the identity element of $P$, then by Peirce decomposition, we can write
	$$P=P_{11}+P_{12}+P_{21}+P_{22}~ \text { where } P_{i j}=e_iPe_j ~ \text{for}~ i, j \in \{1,2\}.$$
	Therefore, each $x \in P$ can be written as $x=a+m+t+b \text { where } a \in P_{11}, m \in P_{12}, t \in P_{21}, b \in P_{22}.$
	First, we discuss a few useful lemmas	before establishing the main results of this paper.
	
	\begin{lemma}
		If $d$ is a commuting Jordan derivation on a ring $P$, then
		\begin{align*}
			d(p^2) = 2d(p)p =2 pd(p),~ \text{for all} ~p \in P.
		\end{align*}
	\end{lemma}
	
	\begin{proof}
		Since $d$ is a commuting Jordan derivation on $P$,
		\begin{align*}
			d(p^2)=d(p)
			p+pd(p)=d(p)p+d(p)p=2d(p)p, ~\text{for all}~p\in P.	\end{align*}
		Similarly, $d(p^2)=2pd(p)$, for all $p\in P$.
	\end{proof}
	\begin{lemma}
		[Lemma $2.1$ \cite{hos}]
		\label{lem1.2}
		Let $P$ be a $2$-torsion free ring and $d : P \rightarrow P$ be a commuting Jordan
		derivation. Then
		for any $p_1, p_2 \in P$, the following holds:		
		\begin{align*}
			& (i)~~ d(p_1 p_2+p_2 p_1) = 2(p_1d(p_2) + p_2d(p_1));\\
			& (ii)~~ d(p_1 p_2+p_2 p_1) = 2(d(p_1)p_2 + d(p_2)p_1);\\
			& (iii)~~ d(p_1p_2p_1) = 3d(p_1)p_2p_1 + d(p_2)p_1^2 -d(p_1)p_1p_2;\\
			& (iv)~~ d(p_1p_2p_1) = p_1^2d(p_2) + 3p_1p_2d(p_1) - p_2p_1d(p_1).
		\end{align*}
		
	\end{lemma}

	\section{Commuting Jordan Derivation}
	\begin{theorem}
		\label{thm2.1}
		Let $P$ be a $2$-torsion free unital ring with a nontrivial idempotent element $e_1$ and $d : P \rightarrow P$ be a commuting Jordan derivation. Then $d$ is a zero map if it satisfies the following condition:	
		\begin{align*}
			[d(a)]_{11} = 0=[d(b)]_{22} , ~\text{for all}~a \in P_{11}~\text{and}~b \in P_{22}.
		\end{align*}

	\end{theorem}
	In order to prove the result, we first establish some useful lemmas.
	\begin{lemma}
		\label{lem2.2}
		$d(e_1) = 0 ~\text{and} ~d(e_2)=0$.	
	\end{lemma}
	\begin{proof}
		
		Since $e_1^2=e_1$,
		\begin{equation}
			\label{eq:2.1}
			\begin{aligned}
				d(e_1 ^2) &= 2e_1 d(e_1) = 2([d(e_1)]_{11} + [d(e_1)]_{12}).
			\end{aligned}
		\end{equation}
		
		Also,
		\begin{equation}
			\label{eq:2.2}
			\begin{aligned}
				d(e_1 ^2)        &= 2d(e_1)e_1 =  2 ([d(e_1)]_{11} + [d(e_1)]_{21}).
			\end{aligned}
		\end{equation}

		Comparing \eqref{eq:2.1} and \eqref{eq:2.2}, we have
		\begin{equation}
			\label{eq:2.3}
			[d(e_1)]_{12} = 0 = [d(e_1)]_{21}.
		\end{equation}
		
		Again,
		\begin{equation}
			\label{eq:2.4}
			\begin{aligned}
				& d\left(e_1\right)=d\left(e_1 \cdot e_1\right)=d\left(e_1\right) e_1+e_1 d\left(e_1\right)\\
				& \implies \left[d\left(e_1\right)\right]_{11}+\left[d\left(e_1\right)\right]_{12}  +\left[d\left(e_1\right)\right]_{21}+\left[d\left(e_1\right)\right]_{22}\\
				&=\left[d\left(e_1\right)\right]_{11}
				+\left[d\left(e_1\right)\right]_{21}+\left[d\left(e_1\right)\right]_{11}+\left[d\left(e_1\right)\right]_{12}\\
				& \implies 	\left[d\left(e_1\right)\right]_{11}-\left[d\left(e_1\right)\right]_{22}=0\\
				& \implies \left[d\left(e_1\right)\right]_{11}=0=\left[d\left(e_1\right)\right]_{22}.
			\end{aligned}
		\end{equation}
		Hence by \eqref{eq:2.3} and \eqref{eq:2.4},  $d(e_1)=0.$\\
		Similarly, we can prove $d\left(e_2\right)=0.$
	\end{proof}
	

	\begin{lemma}
		\label{lem2.3}
		Let $a \in P_{11}$. Then
		$d(a)=[d(a)]_{11}.$
	\end{lemma}

	\begin{proof}
		As $a \in P_{11}$, by Lemma \ref{lem2.2},
		\begin{equation}
			\label{eq:2.5}
			\begin{aligned}
				2d(a)=d\left(a e_1+e_1 a\right) & =2\left(e_1 d(a)+a d(e_1)\right)=2 e_1 d(a).
			\end{aligned}
		\end{equation}
		
		Also, \begin{equation}
			\label{eq:2.6}
			\begin{aligned}
				2d(a)=d\left(a e_1+e_1 a\right) & =2\left(d(a) e_1+d\left(e_1\right) a\right)=2 d(a) e_1.
			\end{aligned}
		\end{equation}
		
		Comparing \eqref{eq:2.5} and \eqref{eq:2.6},
		\begin{equation}
			\label{eq:2.7}
			\begin{aligned}
				& d(a)= e_1 d(a)  = d(a) e_1 ~(\text{using 2-torsion freeness of}~ P)\\
				& \implies [d(a)]_{12}-[d(a)]_{21}=0 \\
				& \implies [d(a)]_{12}=0=[d(a)]_{21}.
			\end{aligned}
		\end{equation}

		By \eqref{eq:2.7},
		\begin{equation}
			\label{eq:2.8}
			d(a)=[d(a)]_{11}+[d(a)]_{22}.
		\end{equation}	
		
		By \eqref{eq:2.5},
		\begin{equation}
			\label{eq:2.9}
			\begin{aligned}
				& d(a)=e_1 d(a)\\
				& \implies [d(a)]_{11}+[d(a)]_{22}=[d(a)]_{11}\\
				& \implies [d(a)]_{22}=0.
			\end{aligned}
		\end{equation}	
		
		Hence, by \eqref{eq:2.7} and \eqref{eq:2.9}, $d(a)=[d(a)]_{11}.$
	\end{proof}
	
	
	\begin{lemma}
		\label{lem2.4}
		Let $m \in P_{12}$. Then
		$d(m)=0.$
	\end{lemma}
	
	\begin{proof}
		Since $m \in P_{12}$,
		\begin{equation}
			\label{eq:2.10}
			\begin{aligned}
				d(m)=d\left(e_1 m+m e_1\right) & =2\left[d\left(e_1\right) m+d(m) e_1\right].
			\end{aligned}
		\end{equation}
		
		Also,
		\begin{equation}
			\label{eq:2.11}
			d(m) =2\left[e_1 d(m)+m d\left(e_1\right)\right].
		\end{equation}
		
		Comparing \eqref{eq:2.10} and \eqref{eq:2.11}, we get
		\begin{equation}
			\label{eq:2.12}
			[d(m)]_{21}=0=[d(m)]_{12}.
		\end{equation}

		From \eqref{eq:2.12}, we have
		\begin{equation}
			\label{eq:2.13}
			d(m)=[d(m)]_{11}+[d(m)]_{22}.
		\end{equation}
		
		Also, from \eqref{eq:2.10},
		\begin{equation}
			\label{eq:2.14}
			\begin{aligned}
				& d(m)=2 d(m) e_1 \\
				\implies & {[d(m)]_{11}+[d(m)]_{22}=2[d(m)]_{11} } \\
				\implies & {[d(m)]_{11}-[d(m)]_{22}=0 }\\
				\implies & {[d(m)]_{11}=[d(m)]_{22}=0 }.
			\end{aligned}
		\end{equation}

		By \eqref{eq:2.12} and \eqref{eq:2.14}, we get $d(m)=0.$
	\end{proof}

	\begin{lemma}
		\label{lem2.5}
		Let $t \in P_{21}$. Then
		$ d(t)=0.$
	\end{lemma}
	
	\begin{proof}
		Since $t \in P_{21}$,	
		\begin{equation}
			\label{eq:2.15}
			\begin{aligned}
				d(t)=d\left(e_2 t+t e_2\right) & =2\left[d\left(e_2\right) t+d(t) e_2\right].
			\end{aligned}
		\end{equation}
		
		Also,
		\begin{equation}
			\label{eq:2.16}
			d(t) =2\left[e_2 d(t)+t d\left(e_2\right)\right].
		\end{equation}
		
		Comparing \eqref{eq:2.15} and \eqref{eq:2.16}, we have
		\begin{equation}
			\label{eq:2.17}
			[d(t)]_{21}=0=[d(t)]_{12}.
		\end{equation}	
		
		From \eqref{eq:2.17},
		\begin{equation}
			\label{eq:2.18}
			d(t)=[d(t)]_{11}+[d(t)]_{22}.
		\end{equation}
		
		Also, from \eqref{eq:2.18},
		\begin{equation}
			\label{eq:2.19}
			\begin{aligned}
				& d(t)=2 e_2d(t) \\
				\implies & {[d(t)]_{11}+[d(t)]_{22}=2[d(t)]_{22} } \\
				\implies & {[d(t)]_{11}-[d(t)]_{22}=0 }\\
				\implies & {[d(t)]_{11}=[d(t)]_{22}=0 }.
			\end{aligned}
		\end{equation}

		By \eqref{eq:2.17} and \eqref{eq:2.19}, we obtain $d(t)=0.$
	\end{proof}

	\begin{lemma}
		\label{lem2.6}
		Let $b \in P_{22}$. Then
		$d(b)=[d(b)]_{22}.$
	\end{lemma}

	\begin{proof}
		As $b \in P_{22}$, by Lemma \ref{lem2.2},
		\begin{equation}
			\label{eq:2.20}
			\begin{aligned}
				2d(b)=d\left(b e_2+e_2 b\right) & =2\left(e_2 d(b)+b d(e_2)\right)=2 e_2 d(b).
			\end{aligned}
		\end{equation}
		
		Also, \begin{equation}
			\label{eq:2.21}
			\begin{aligned}
				2d(b)=d\left(b e_2+e_2 b\right) & =2\left(d(b) e_2+d\left(e_2\right) b\right)=2 d(b) e_2.
			\end{aligned}
		\end{equation}
		
		Comparing \eqref{eq:2.20} and \eqref{eq:2.21}, we get
		\begin{equation}
			\label{eq:2.22}
			\begin{aligned}
				& d(b)= e_2 d(b)  = d(b) e_2 ~(\text{using 2-torsion freeness of}~ R)\\
				& \implies [d(b)]_{12}-[d(b)]_{21}=0 \\
				& \implies [d(b)]_{12}=0=[d(b)]_{21}.
			\end{aligned}
		\end{equation}

		By \eqref{eq:2.22},
		\begin{equation}
			\label{eq:2.23}
			d(b)=[d(b)]_{11}+[d(b)]_{22}.
		\end{equation}	
		
		By \eqref{eq:2.20} and \eqref{eq:2.23}, we have
		\begin{equation}
			\label{eq:2.24}
			\begin{aligned}
				& d(b)=e_2 d(b)\\
				& \implies [d(b)]_{11}+[d(b)]_{22}=[d(b)]_{22}\\
				& \implies [d(b)]_{11}=0.
			\end{aligned}
		\end{equation}	
		
		Hence, by \eqref{eq:2.22} and \eqref{eq:2.24}, $d(b)=[d(b)]_{22}.$
	\end{proof}
	Now, we prove Theorem \ref{thm2.1}.
	\begin{proof}[\textbf {Proof of Theorem \ref{thm2.1}}]
		Let $x\in P$. Then
		\begin{align*}
			x=a+m+t+b, ~\text{for some}~ a\in P_{11},~m\in P_{12},~t\in P_{21} ~\text{and}~b\in P_{22}.
		\end{align*}
		
		By using Lemmas \ref{lem2.3}, \ref{lem2.4}, \ref{lem2.5} and \ref{lem2.6}, we have
		\begin{align*}
			d(x)=[d(a)]_{11}+[d(b)]_{22}=0~\text{(by the assumption)}.
		\end{align*}
		Hence, $d=0$.
	\end{proof}
	
	\begin{cor}
		\label{cor2.7}
		Let $P$ be a $2$-torsion free prime ring with a non-trivial idempotent $e_1$ and $d : P \rightarrow P$ be a commuting Jordan derivation. Then, $d$ is a zero map.
	\end{cor}
	
	\begin{proof}Let $x\in P$. Then
		\begin{align*}
			x=a+m+t+b, ~\text{for some}~ a\in P_{11},~m\in P_{12},~t\in P_{21} ~\text{and}~b\in P_{22}.
		\end{align*}
		
		Using the proof of Theorem \ref{thm2.1}, we have
		\begin{equation}
			\label{eq:2.25}
			d(x)=[d(a)]_{11}+[d(b)]_{22}.
		\end{equation}
		
		By Lemmas \ref{lem1.2}, \ref{lem2.3} and \ref{lem2.5}, we get
		\begin{equation}
			\label{eq:2.26}
			\begin{aligned}
				& d(at+ta)=2td(a)=2d(a)t\\
				& \implies td(a)=d(a)t ~\text{(Since}~P ~\text{is 2-torsion free)}\\
				& \implies t[d(a)]_{11}=[d(a)]_{11}t=0\\
				& \implies e_2R[d(a)]_{11}=0\\
				& \implies [d(a)]_{11}=0 ~\text{(Since}~P ~\text{is prime and}~ e_2\neq 0).
			\end{aligned}
		\end{equation}
		
		Also, by Lemmas \ref{lem1.2}, \ref{lem2.4} and \ref{lem2.6}, we have
		\begin{equation}
			\label{eq:2.27}
			\begin{aligned}
				& d(mb+bm)=2md(b)=2d(b)m\\
				& \implies md(b)=d(b)m ~\text{(Since}~P ~\text{is 2-torsion free)}\\
				& \implies m[d(b)]_{22}=[d(b)]_{22}m=0\\
				& \implies e_1R[d(b)]_{22}=0\\
				& \implies [d(b)]_{22}=0 ~\text{(Since}~P ~\text{is prime and}~ e_1\neq 0).
			\end{aligned}
		\end{equation}
		
		Again, from \eqref{eq:2.25}, \eqref{eq:2.26} and \eqref{eq:2.27}, we have
		\begin{align*}
			d(x)=0 ~\text{for all}~ x\in P.
		\end{align*}
		
		Hence, $d$ is a zero map.
	\end{proof}

	\begin{cor}
		Let $P$ be a 2-torsion free prime ring with a non-trivial idempotent $e_{1}$ such that ${\{d_n\}}_{n=0}^{\infty}$ is a commuting Jordan higher derivation on $P,$ that is, $d_n (x^2) = \sum_{k=0}^{n} d_{n-k}(x)d_k (x)$, $d_0(x)=x$ and $[d_n(x),x]=0$ $([a,b]=ab-ba)$, for all non-negative integers $n$ and $x \in P$. Then $d_n =0$ for all $n \in \mathbb{N}$.
	\end{cor}
	
	\begin{proof}
		We use induction on $n$ to prove this result. For $n=1$,
		\begin{align*}
			d_1 (x^2) = d_1 (x)d_0(x) +  d_0(x)d_1 (x) = d_1 (x)x + xd_1 (x),~\text{for all}~x\in P.	
		\end{align*}
		Hence, $d_1$ is a Jordan derivation on $P$. Since $[d_1(x), x]= 0$ for all $x \in P$, $d_1$ is a commuting Jordan derivation on $P$. It follows from Corollary \ref{cor2.7} that $d_1$ is a zero map.
		For $n=2$,
		\begin{align*}
			d_2 (x^2) = d_2 (x)x + (d_1 (x))^2 + xd_2 (x) = d_2 (x)x + xd_2 (x),~\text{for all}~x\in P.	
		\end{align*}
		Therefore, $d_2$ is a Jordan derivation on $P$. Since $[d_2(x), x]= 0$ for all $x \in P$, $d_2$ is a commuting Jordan derivation on $P$. It follows from Corollary \ref{cor2.7} that $d_2$ is a zero map.
		Let $n$ be an arbitrary positive integer and assume that the result holds for any $k < n$. To prove the result for $n$, we have
		\begin{align*}
			d_n (x^2) = \sum_{k=0}^{n} d_{n-k}(x)d_k (x)  = d_n (x)x + xd_n (x),~\text{for all}~x\in P.
		\end{align*}
		Hence, $d_n$ is a Jordan derivation on $P$. Since $[d_n(x), x]= 0$ for all $x \in P$, $d_n$ is a commuting Jordan derivation on $P$. It follows from Corollary \ref{cor2.7} that $d_n$ is a zero map.
		
	\end{proof}
	
	\section{Commuting Generalized Jordan Derivations Over a Ring}
	\begin{theorem}
		\label{thm3.1}
		If $D$ is a commuting Jordan left centralizer over a $2$-torsion free prime ring $P$ with a nontrivial idempotent $e_1$, then $D$ is a zero map on $P$ if it satisfies the following condition:
		$$[D(m)]_{12} = 0 ~\text{for ~all} ~m \in P_{12}.$$
	\end{theorem}

	Since $D$ is commuting Jordan centralizer over a $2$-torsion free prime ring $P$, then by Proposition $2.5$ of \cite{bre0}, $D$ becomes a two-sided centralizer. We use this to prove the following lemmas.
	
	\begin{lemma}
		\label{lem3.1}
		For any $m \in P_{12}$, we have $D(m) =0.$
	\end{lemma}	
	\begin{proof}
		Let $m \in P_{12}$. Then
		\begin{equation}
			\label{eq:3.1}
			D(m) = D(me_2) =D(m)e_2 =[D(m)]_{12} + [D(m)]_{22}.
		\end{equation}
		Also, \begin{equation}
			\label{eq:3.2}
			D(m) =D(e_1m)= e_1D(m) =[D(m)]_{11} + [D(m)]_{12}.
		\end{equation}
		Therefore, from \eqref{eq:3.1} and \eqref{eq:3.2}, we obtain
		\begin{equation}
			\label{eq:3.0}
			[D(m)]_{11} = 0 = [D(m)]_{22}.
		\end{equation}
		Now, \begin{align*}
			&D(m)=D(me_2)\\
			&\implies [D(m)]_{12} +[D(m)]_{21} = [D(m)]_{12} + [D(m)]_{22}\\
			&\implies [D(m)]_{21} = 0 ~\text{(by~ using~ equation ~\eqref{eq:3.0})}.
		\end{align*}
		By assumption, $[D(m)]_{12} = 0$ for all $m \in P_{12}.$ \\
		Thus, $D(m) =0$ for all $m \in P_{12}.$	
	\end{proof}	
	
	\begin{lemma}
		\label{lem3.4}
		For any $a \in P_{11}$, we have $D(a) =0.$
	\end{lemma}	
	
	\begin{proof}
		Since $a \in P_{11}$, we have
		\begin{equation}
			\label{eq:3.3}
			D(a) = D(ae_1) =D(a)e_1 =[D(a)]_{11} + [D(a)]_{21}.
		\end{equation}
		Also, \begin{equation}
			\label{eq:3.4}
			D(a) = D(e_1a) = e_1D(a) =[D(a)]_{11} + [D(a)]_{12}.
		\end{equation}
		
		Therefore, from \eqref{eq:3.3} and \eqref{eq:3.4}, we obtain
		\begin{equation}
			\label{eq:3.06}
			[D(a)]_{12} = 0 = [D(a)]_{21}.
		\end{equation}
		Now,
		\begin{align*}
			&D(a)=D(ae_1)\\
			&\implies [D(a)]_{11} +[D(a)]_{22} = [D(a)]_{11}\\
			&\implies [D(a)]_{22} = 0.
		\end{align*}
		
		Also, for $m \in P_{12}$, we obtain
		\begin{align*}
			&D( am + ma) = aD(m)  + mD(a)\\
			& \implies D( am + ma) = 0 + m[D(a)]_{11} ~\text{(by ~using ~Lemma ~\ref{lem3.1})}\\
			& \implies D(am) =0\\
			& \implies D(a) m =0\\
			& \implies [D(a)]_{11} me_2 =0~\text{(for~ all}~ m \in P_{12})\\
			& \implies [D(a)]_{11}Pe_2 =0\\
			& \implies [D(a)]_{11} =0 ~\text{(since ~$P$~ is ~prime)}.
		\end{align*}
		Thus, by equation \eqref{eq:3.06}, $D(a) =0$ for all $a \in P_{11}.$ 	
	\end{proof}

	\begin{lemma}
		\label{lem3.3}
		For all $t \in P_{21}$, we have $D(t) =0.$
	\end{lemma}	
	
	\begin{proof}
		Since $t \in P_{21}$,
		\begin{equation}
			\label{eq:3.5}
			D(t) = D(te_1) =D(t)e_1 =[D(t)]_{11} + [D(t)]_{21}.
		\end{equation}
		Now, \begin{equation}
			\label{eq:3.6}
			D(t) = D(e_2t) =e_2D(t) =[D(t)]_{21} + [D(t)]_{22}.
		\end{equation}	
		
		Therefore, from \eqref{eq:3.5} and \eqref{eq:3.6}, we obtain $[D(t)]_{11} = 0 = [D(t)]_{22}$.\\
		Also, \begin{align*}
			&D(t)=D(te_1)\\
			&\implies [D(t)]_{12} +[D(t)]_{21} = [D(t)]_{21}\\
			&\implies [D(t)]_{12} = 0.
		\end{align*}
		Again, for $a \in P_{11}$, we have
		\begin{align*}
			&D(at + ta)= a D(t)+ tD(a)\\
			& \implies D(ta) =a[D(t)]_{21}+0 ~\text{(by ~using ~Lemma ~\ref{lem3.4})}\\
			& \implies D(t)a =0 \\
			& \implies [D(t)]_{21}a =0~\text{(for~ all}~ a \in P_{11})\\
			& \implies [D(t)]_{21}Pe_1 =0\\
			& \implies [D(t)]_{21} =0 ~\text{(since ~$P$~ is ~prime)}.
		\end{align*}
		Hence, $D(t) =0$ for all $t \in P_{21}.$
	\end{proof}

	\begin{lemma}
		\label{lem3.5}
		For any $b \in P_{22}$, we have $D(b) =0.$
	\end{lemma}	
	
	\begin{proof}
		As $b \in P_{22}$, we get
		\begin{equation}
			\label{eq:3.7}
			D(b) = D(be_2) =D(b)e_2 =[D(b)]_{12} + [D(b)]_{22}.
		\end{equation}
		Also, \begin{equation}
			\label{eq:3.8}
			D(b) = D(e_2b) = e_2D(b) =[D(b)]_{21} + [D(b)]_{22}.
		\end{equation}		
		
		Therefore, from \eqref{eq:3.7} and \eqref{eq:3.8}, we obtain $[D(b)]_{12} = 0 = [D(b)]_{21}$.\\
		Now, \begin{align*}
			&D(b)=D(be_2) =D(b)e_2\\
			&\implies [D(b)]_{11} +[D(b)]_{22} = [D(b)]_{22}\\
			&\implies [D(b)]_{11} = 0.
		\end{align*}
		Also, for $t \in P_{21}$, we have
		\begin{align*}
			&D( bt + tb) = bD(t)  + tD(b)\\
			& \implies D(bt + tb) = 0 + t[D(b)]_{22}~\text{(by ~using~Lemma ~\ref{lem3.3})}\\
			& \implies D(bt) =0\\
			& \implies D(b) t =0\\
			& \implies [D(b)]_{22} te_1 =0~\text{(for~ all}~ t \in P_{21})\\
			& \implies [D(b)]_{22} Pe_1 =0\\
			& \implies [D(b)]_{22} =0 ~\text{(since ~$P$~ is ~prime)}.
		\end{align*}
		Thus, $D(b) =0$ for all $b \in P_{22}.$ 	
	\end{proof}		
	
	Now, we are ready to prove Theorem \ref{thm3.1}.	
	\begin{proof}[\textbf {Proof of Theorem \ref{thm3.1}}]
		Let $x\in P$. Then
		\begin{align*}
			x=a+m+t+b, ~\text{for some}~ a\in P_{11},~m\in P_{12},~t\in P_{21} ~\text{and}~b\in P_{22}.
		\end{align*}
		
		By using Lemmas \ref{lem3.1}, \ref{lem3.4}, \ref{lem3.3} and \ref{lem3.5}, we obtain
		\begin{align*}
			D(x)=D(a)+ D(m)+ D(t)+ D(b)=0, ~\text{for~ all}~ x \in P.
		\end{align*}
		Thus, $D=0$.
	\end{proof}

	\begin{cor}
		If $D$ is commuting generalized Jordan derivation over a $2$-torsion free prime ring $P$ with a non-trivial idempotent $e_1$, then $D$ is a zero map if it satisfies the following condition:
		$$[D(m)]_{12} = 0 ~\text{for ~all} ~m \in P_{12}.$$
	\end{cor}	
	
	\begin{proof}
		Since $D$ is a commuting generalized Jordan derivation, there is an associated commuting Jordan derivation $d$ over the ring $P$. Then, by Theorem \ref{thm2.1}, $d$ is a zero map. As a result, $D$ becomes a commuting Jordan left centralizer and hence by Theorem \ref{thm3.1}, $D$ is a zero map on $P$.
	\end{proof}

	\section*{Acknowledgment}
	The first author is thankful to the University Grants Commission (UGC), Govt. of India, for financial support under UGC Ref. No. 1256 dated 16/12/2019. The authors are also thankful to the Indian Institute of Technology Patna for providing research facilities.
	
	\section*{Declarations}
	\textbf{~~~Data Availability Statement}: The authors declare that [the/all other] data supporting the findings of this study are available within the article. Any clarification may be requested from the corresponding author, provided it is essential. \\\\

	\textbf{Competing interests}: The authors declare that there is no conflict of interest regarding the publication of this manuscript.\\
	
	\textbf{Use of AI tools declaration}
	The authors declare that they have not used Artificial Intelligence (AI) tools in the creation of this manuscript.

	\end{document}